\documentclass{article}
\usepackage{lmodern}
\usepackage[T1]{fontenc}
\usepackage[utf8]{inputenc}
\usepackage{amsmath}
\usepackage{graphicx}
\usepackage{esvect}
\usepackage{hyperref}
\usepackage{MnSymbol}
\usepackage{amsfonts}
\usepackage{float}
\usepackage{amsthm}
\usepackage{mathtools}
\usepackage{systeme} 
\usepackage{bbold}
\usepackage{relsize}
\usepackage{mathtools}
\usepackage{fancyhdr}
\usepackage{dcolumn}
\usepackage{color}
\usepackage{authblk}
\usepackage[
]{biblatex}
\newcolumntype{2}{D{.}{}{2.0}}

\DeclareMathOperator{\rank}{\mathrm{rank}}
\DeclareMathOperator{\diver}{\mathrm{div}}

\newtheorem{theorem}{Theorem}[section]

\newtheorem{defi}{Definition}[section]
\newtheorem{prop}{Proposition}[section]
\newtheorem{corollary}{Corollary}[section]
\newtheorem{lemma}{Lemma}[section]
\numberwithin{equation}{section}
\usepackage[utf8]{inputenc}
\title{Surfaces of genus $g\geq 1$ in 3D contact sub-Riemannian manifolds}
\author[1,2]{Eugenio Bellini}
\author[3]{Ugo Boscain}
\affil[1]{SISSA, Via Bonomea 265, 34136 Trieste, Italy. eugenio.bellini@outlook.it}
\affil[2]{Dipartimento di Matematica e Applicazioni, Università degli Studi di Milano-Bicocca, Milano, Italy.}
\affil[3]{CNRS, Laboratoire Jacques-Louis Lions, Sorbonne
Université, Université de Paris, Inria, Boîte courrier 187, 75252
Paris Cedex 05 Paris, France. ugo.boscain@sorbonne-universite.fr}

\date{}

\begin{document}
\maketitle
\begin{abstract}
 We consider smooth embedded surfaces in a 3D contact sub-Riemannian manifold and the problem of the finiteness of the induced distance (i.e., the infimum of the 
 length of horizontal curves that belong to the surface). Recently it has been proved that for a surface having the topology of a sphere embedded in a tight co-orientable structure, the distance is always finite. In this paper we study closed surfaces of genus larger than 1, proving that such surfaces can be embedded in such a way that the induced distance is finite or infinite. We then study the structural stability of the finiteness/not-finiteness of the distance.
\end{abstract} 
\section{Introduction}
Consider a three dimensional contact sub-Riemannian manifold $(M,\mathcal D, \bf{g})$, where $M$ is a smooth manifold of dimension three, $\mathcal D$ is a contact distribution (i.e., such that $\rank{\mathcal D}=2$ and $\mathcal D + \left[\mathcal D,\mathcal D\right]=TM$) and $\bf{g}$ is a sub-Riemannian metric \cite{1, 12}. 
Let $S$ be a smooth surface embedded in $M$. The intersection $\mathcal D\cap TS$ defines a field of directions with singularities on $S$, where the singularities are the points $q\in S$ such that $\mathcal D_q=T_qS$. Such points are called \emph{characteristic points}. In the following the set of characteristic points is denoted with $\Sigma (S)$. The set of integral curves of $\mathcal D\cap TS$ on $S\setminus \Sigma(S)$ is called the \emph{characteristic foliation} of $S$, and it is denoted with $\mathcal D S$. For an example of characteristic foliation of a surface in the Heisenberg group see Figure \ref{sphere_H}.
\begin{figure}[H]
\begin{center}
    \includegraphics[width=12cm]{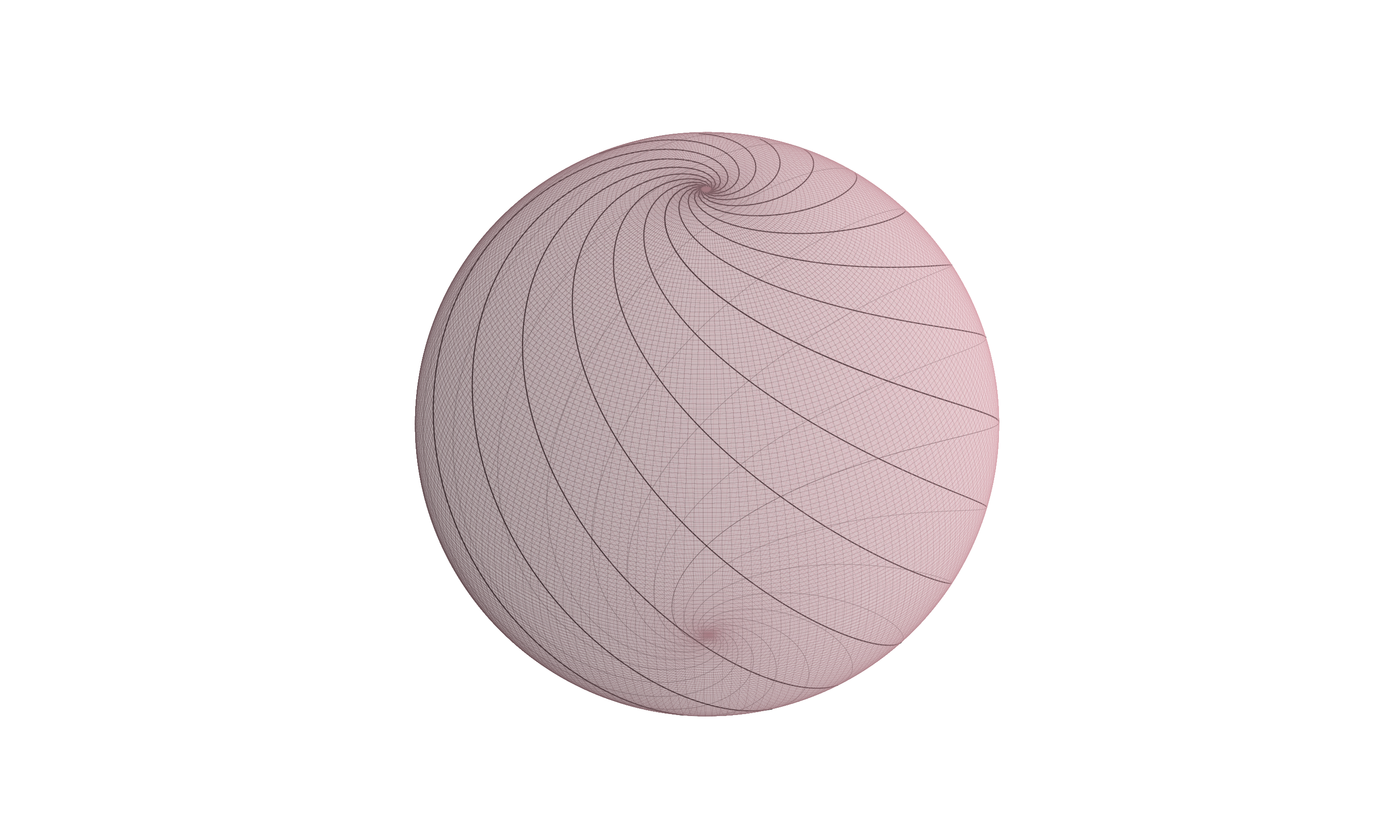}
\end{center}
\caption{Characteristic foliation of the Euclidean unit sphere centered at the origin, embedded in the Heisenberg group $(\mathbb R^3, \mathcal D=\text{span}\{\partial_x+\frac{y}{2}\partial_z, \partial_y-\frac{x}{2}\partial_z\})$.}
\label{sphere_H}
\end{figure}
Each leaf of the characteristic foliation on $S\setminus\Sigma(S)$ is a one-dimensional Riemannian manifold, the metric being the restriction of $\bf{g}$. A lipschitz curve $\gamma:[0,T]\to S$ is called \emph{admissible} if $\dot\gamma(t)\in\mathcal D_{\gamma(t)}$ for almost every $t\in[0,T]$.
We can use the metric $\bf{g}$ to measure the length of admissible curves on $S$
\begin{equation*}
L(\gamma):=\int_{0}^T\sqrt{{\bf{g}}(\dot\gamma(t),\dot\gamma(t))}\,dt,
\end{equation*}
and we define the \emph{induced distance} between two points $q_0, q_1\in S$ as 
\begin{equation}
    d_{S}(q_0,q_1)=\inf\{\,L(\gamma)\,:\,\gamma\,\,\text{admissible},\,\gamma(0)=q_0,\,\gamma(T)=q_1\,\},
\end{equation}
with the understanding that the induced distance is infinite whenever there is no admissible curve connecting $q_0$ to $q_1$. 
In this paper we are interested in studying the finiteness of $d_S$. 
We say that the induced distance 
is finite if it is finite for every pair of points, in that case the couple $(S,d_S)$ defines a metric space and in particular a length space. Notice that this distance does not coincide with the restriction of the sub-Riemannian distance on $S$, moreover $d_S$ is never continuous with respect to the topology of the surface \cite{5}. 
If the sub-Riemannian manifold admits a global orthonormal frame $(F_1,F_2)$, or equivalently if $\mathcal D$ is a trivial vector bundle over $M$, then the problem of finding the 
curve realizing the distance between two points $q_0$ and $q_1$ can be written as the optimal control problem
\begin{align}
&\dot\gamma(t)=u_1(t)F_1(\gamma(t))+u_2(t)F_2(\gamma(t)),\nonumber\\
&\gamma(0)=q_0,~~~\gamma(T)=q_1,\nonumber\\
&\int_0^T \sqrt{u_1^2(t)+u_2(t)^2}\,dt\to\min,\nonumber\\
&\mbox{with the state constraint $\gamma(t)\in S$.}\nonumber
\end{align}

In the study of contact distributions on 3-manifolds from a topological perspective, embedded surfaces play a central role \cite{6, 10, 11}. 
Recently related topics have been studied when the contact distribution is endowed with a Riemannian metric (surfaces embedded in 3D contact sub-Riemannian manifolds). See for instance \cite{7,8} for Carnot groups, \cite{5} for generic structures, \cite{2,3,15} for Gauss-Bonnet theorems
and \cite{4} for stochastic evolution equations. 
In particular in \cite{5} the authors proved that the induced distance is always finite for spheres embedded in a co-oriented tight  3D-contact sub-Riemannian manifolds. Recall that a 3D-contact sub-Riemannian manifold $(M,\mathcal D,\bf{g})$ is co-orientable if the distribution $\mathcal D$ can be globally expressed as the kernel of a 1-form $\omega$. This form is called the contact form, and in the context of sub-Riemannian geometry is normalized in such a way that
\begin{equation}\label{normalized_form}
d\omega_{|\mathcal D}=\text{Vol}_\textbf{g},
\end{equation}
where $\text{Vol}_\textbf{g}$ is the area form induced on $\mathcal D$ by the metric $\bf{g}$.
The contact condition $\mathcal D+[\mathcal D,\mathcal D]=TM$ can be expressed in term of the contact form $\omega$ as $\omega\wedge d\omega \neq 0$, thus a co-orientable contact manifold $M$ is 
necessarily orientable.
Recall moreover that a contact distribution is called overtwisted (see 4.5 of \cite{9}) if admits an overtwisted disk, i.e., an embedding of a disk with horizontal boundary whose characteristic foliation has a unique singular point: an elliptic point in the interior of the disk, see Figure \ref{ot-disk}.
\begin{figure}[H]
\begin{center}
    \includegraphics[width=11cm]{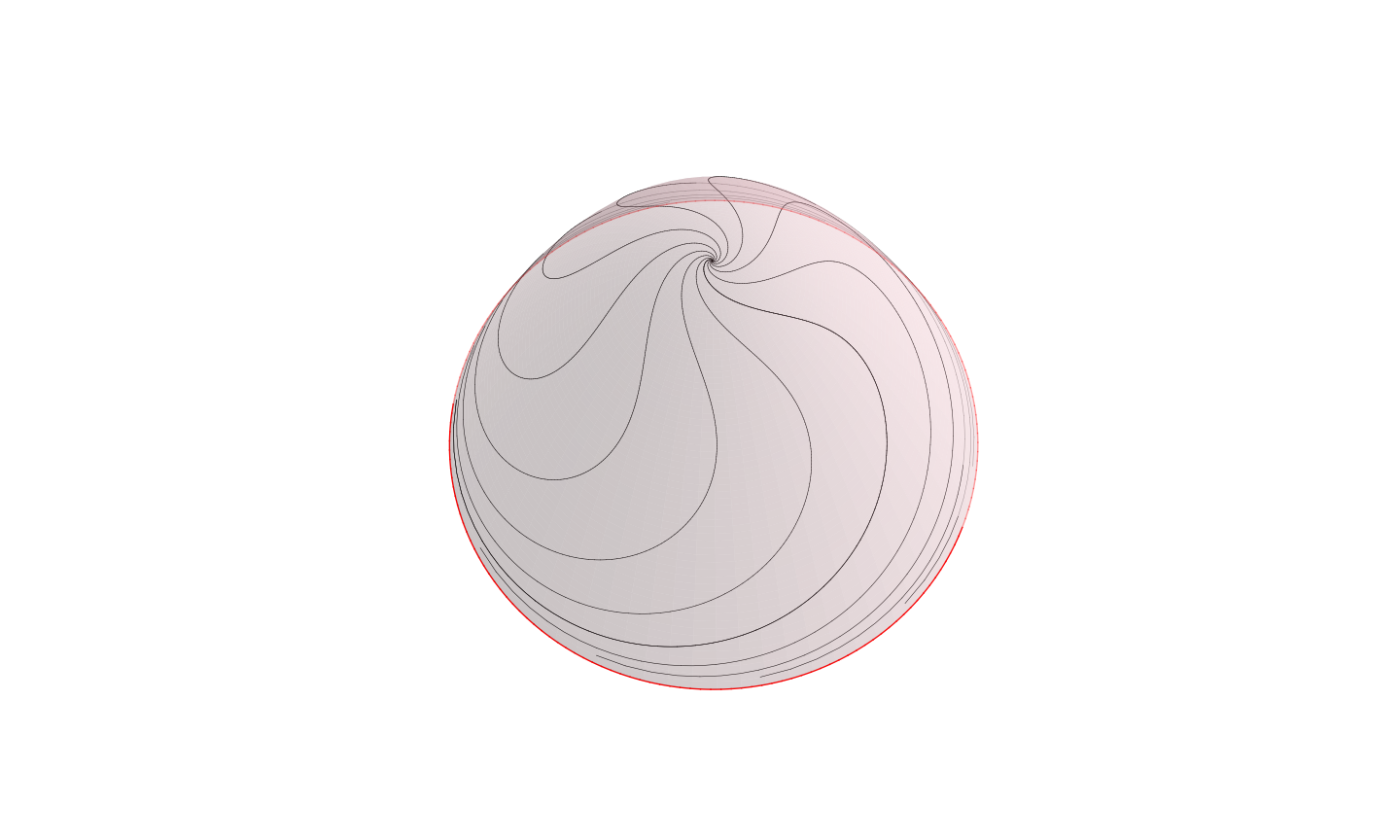}
\end{center}
\caption{The image depicts the overtwisted disk $\{\,(r,\theta,z)\,:\,z=\frac{1}{2}(\pi^2-r^2),\,r\leq \pi,\,\theta \in [0,2\pi)\,\}$ in the contact structure $(\mathbb R^3, \mathcal D=\text{span}\{\partial_r, \frac{\cos r}{r}\partial_\theta-\sin r\partial_z\})$, expressed in cylindrical coordinates $(r,\theta, z)$. Notice that, even though the field $\frac{\cos r}{r}\partial_\theta-\sin r\partial_z$ is not continuous in $r=0$, the distribution $\mathcal D=\ker\{\cos rdz+r\sin rd\theta\}$ is well defined.}
\label{ot-disk}
\end{figure}
A contact distribution is called tight if it is not overtwisted. Since every overtwisted contact structure contains a sphere $S$ having closed orbits in its characteristic foliation, which necessarily implies that $d_S$ is not finite, the result obtained in \cite{5} can be stated as a metric characterization of tightness.
\begin{theorem}\label{thm-tightness}\cite{5}
    Let $(M,\mathcal D,\bf{g})$ be a co-oriented $3D$-contact sub-Riemannian manifold, then $(M,\mathcal D)$ is tight if and only if every embedded surface homemorphic to a sphere has finite induced distance.
\end{theorem}
In this paper we prove that the characterization given by Theorem \ref{thm-tightness} works only with spheres. In particular we obtain the following result.
\begin{corollary}
Let $(M,\mathcal D,\bf{g})$ be a co-orientable 3D-contact sub-Riemannian structure. Let $S$ be a closed orientable surface of genus $g\geq 1$. Then there exist an embedding of $S$ in $M$ for which the induced distance is finite, and one for which the induced distance is not finite.
\end{corollary}
The latter result is a corollary of the following theorem, stating that generically the finiteness (not finitess) of the distance is a structurally stable property.
\begin{theorem}\label{main}
Let $(M,\mathcal D,\bf{g})$ be a co-orientable 3D-contact sub-Riemannian manifold, let $S$ be a closed orientable surface of genus $g\geq 1$, and let $\textbf{Emb}(S,M)$ be the space of embeddings of $S$ in $M$ endowed with the $C^{\infty}$-topology. Then there exist two disjoint non empty open subsets, $\mathcal U_f,\mathcal U_{\infty}\subset \textbf{Emb}(S,M)$, having dense union, such that the induced distance is finite on $\mathcal U_f$ and not finite in $\mathcal U_{\infty}$.
\end{theorem}
The proof of this theorem is largely based on the theory of Morse-Smale vector fields, which are structurally stable. In Section \ref{sec-Morse-Smale} we recall the basic facts concerning these vector fields. We prove that the only obstruction to the finiteness of the induced distance on Morse-Smale characteristic foliations is the existence of closed orbits. We then recall the elimination lemma, a result from contact topology which we use to break up the closed orbits of a Morse-Smale foliation by means of an arbitrarily small $C^0$ perturbation. These facts together ensure the existence of the set $\mathcal U_f$ of Theorem \ref{main}.
In Section \ref{sec-U-infty} we prove the existence of the set $\mathcal U_\infty$ and we conclude the proof of Theorem \ref{main}. 
\section{Morse-Smale vector fields}\label{sec-Morse-Smale}
The characteristic foliation of an oriented surface $S$ embedded in a co-oriented contact manifold $(M, \omega)$ is driven by a global vector field: there exists a smooth vector field $X$ on $S$ vanishing on $\Sigma (S)$, satisfying
\begin{equation}
    T_{q}S\cap \mathcal D_q=
   \text{span}_{\mathbb R} X_q,\,\,\forall\,q\in S\setminus \Sigma(S).
\end{equation}
A vector field with the properties described above is called a characteristic vector field of $S$. 
Characteristic vector fields are not unique, and can be obtained one from the other via multiplication by a positive function. 
In particular the choice of a characteristic vector field corresponds to the choice of an area form $\Omega$ on the orientable surface $S$, compatible with the orientation of the latter. Indeed, once the area form $\Omega$ is chosen, the corresponding characteristic vector field is the unique solution to the equation
\begin{equation}
    i_{X}\Omega=\omega_{|S}.
\end{equation}
The contact condition $\omega\wedge d\omega>0$ implies that the divergence $\diver_{\Omega}X(q)$ is non zero for every $q\in\Sigma(S)$:
\begin{equation}\label{div_eq}
    d\circ i_X\Omega_q=\diver_{\Omega}X(q)\Omega_q\neq 0.
\end{equation}
The sign of a characteristic point $q\in \Sigma(S)$ is defined as
\begin{equation}
\text{sign}(q):=\text{sign}(\diver_{\Omega }X(q)).
\end{equation}
For a generic surface $S$, the characteristic vector field is Morse-Smale: a type of vector field with particularly simple dynamical features, recalled in the definition below. 
 \begin{defi}\label{def-Morse}
    Let $S$ be a closed orientable surface, a vector field $X$ on $S$ is called Morse-Smale if 
    \begin{itemize}
        \item $X$ has finitely many critical points and closed orbits, all of which are non degenerate,
        \item the $\alpha$-limit of every trajectory is either a critical point or a closed orbit, and the same holds for $\omega$-limits,
        \item there are no saddle connections.
    \end{itemize}
\end{defi}
In its celebrated stability theorems, M.M. Peixoto  showed that such vector fields are generic and structurally stable on closed orientable surfaces (\cite{14},\cite{13}). These results are summarized in the following theorem.
\begin{theorem}
   Let $X$ be a Morse-Smale vector field on a closed orientable surface $S$. The dynamics of any vector field $X'$ sufficiently $C^1$-close to $X$ is topologically conjugated to the dynamics of $X$. The homeomorphism realizing the topological equivalence can be chosen $C^0$-close to the identity. Morse-Smale vector fields of class $C^r$, with $1\leq r \leq \infty$, on a closed orientable surface form an open and dense set of the set o vector fields on $S$ endowed with the $C^r$-topology.
\end{theorem}
Actually, Morse-Smale vector fields are generic among the characteristic vector fields of closed orientable surfaces in contact 3-manifolds, as it follows from the following proposition. 
\begin{prop}\label{first-piece}[Proposition 4.6.11 of \cite{9}]
    Let $S$ be a closed orientable surface embedded in a co-oriented contact manifold.Then, there exists a surface $S'$, isotopic and $C^\infty$ close to $S$, having a Morse-Smale characteristic foliation.
\end{prop}
A necessary condition for a surface to have finite induced distance is the absence of closed orbits. For a surface with a Morse-Smale characteristic foliation this condition is also sufficient.
 \begin{lemma}\label{second-piece}
    Let $S$ be a closed orientable surface embedded in a co-oriented 3D-contact sub-Riemannian manifold $M$, having characteristic foliation driven by a Morse-Smale vector field $X$. Then the induced distance $d_{S}$ is finite if and only if $X$ does not have closed orbits.
\end{lemma}
\begin{proof}
     As we have already said, if $X$ has closed orbits, then the induced distance is not finite. Assume now that $X$ does not have closed orbits.
     The Morse-Smale property ensures that any leaf of the characteristic foliation is admissible, and hence has finite length: the finiteness of these lengths is proved in \cite{5} Proposition 1.3, and it is a consequence of the fact that the integral curves of $X$ contained in the stable manifold of a non-degenerate critical point $q$ converge sub-exponentially to $q$.
    Denoting with $\mathcal A_{x,y}$ the set of admissible curves joining $x$ to $y$, we only need to show that $\mathcal A_{x,y}\neq\emptyset$ for any $x,y\in S$.
    Let $ S'$ be the complement of the hyperbolic points (saddles) in $S$, which is open and connected. By definition of Morse-Smale vector field, every saddle is connected to some elliptic point, thus it is sufficient to show that $\mathcal A_{x,y}\neq\emptyset$ for any $x,y\in S'$.
    For every $x,y\in S$ we define the equivalence relation 
    \begin{equation}
        x\sim y \iff \mathcal A_{x,y}\neq\emptyset,
    \end{equation}
    and we denote with $[x]$ the equivalence class of $x$.
    By definition of Morse-Smale vector field, for every $x\in S'$ the limits
    \begin{equation}\label{critical_lim}
        \omega(x)=\lim_{t\to\infty}e^{tX}(x),\,\,\,\alpha(x)=\lim_{t\to\infty}e^{-tX}(x),
    \end{equation}
    exist and are critical points of $X$. Moreover at least one of the two limits must be an elliptic point, because there are no trajectories joining hyperbolic points. Therefore for every $x\in S'$ the set $[x]\cap S'$ contains an open neighbourhood of $x$, which is the stable (unstable) manifold of some elliptic point. It follows that the equivalence classes are open, thus, $S'$ being connected, there is only one equivalence class.
\end{proof}
Given any orientable surface $S$ embedded in a 3D-contact manifold, with an arbitrarily small $C^\infty$ perturbation, we can achieve an isotopic surface with characteristic foliation of Morse-Smale type. The new surface may well present closed orbits, even if the original surface $S$ did not have any. The elimination lemma, a result due to Giroux which we are about to state, describes a procedure which allows us to destroy closed orbits preserving the Morse-Smale property, by means of an arbitrarily $C^0$-small perturbation of the surface, see Figure \ref{b_el} and \ref{a_el}.
\begin{lemma}\label{elimination_lem}[Lemma 4.6.26 of \cite{9}]
    Let $(M,\mathcal D)$ be a three dimensional co-orientable contact manifold and $S$ and embedded closed oriented surface with characteristic foliation of Morse-Smale type. Assume that there exist two singular point of the same sign, one elliptic point $q_e$ and one hyperbolic point $q_h$, connected by a a separatrix $\gamma$ of the latter. Let $U$ be an arbitrarily small neighbourhood of $\gamma$. Then there exists an isotopy $\varphi_t:S\to M$, $t\in[0,1]$, with the following properties:
    \begin{itemize}
        \item $\varphi_0:S\to M$ is the inclusion of $S$ in $M$,
        \item $\varphi_t$ can be chosen arbitrarily $C^0$-close to $\varphi_0$,
        \item The isotopy is fixed on $\gamma$ and outside $U$,
        \item $\varphi_1(S)$ has a characteristic foliation of Morse-Smale type,
        \item the characteristic foliation of $\varphi_1(S)$ has no singularities in $\varphi_1(U)$.
    \end{itemize}
\end{lemma}
\begin{figure}[H]
\begin{center}
    \includegraphics[width=11cm]{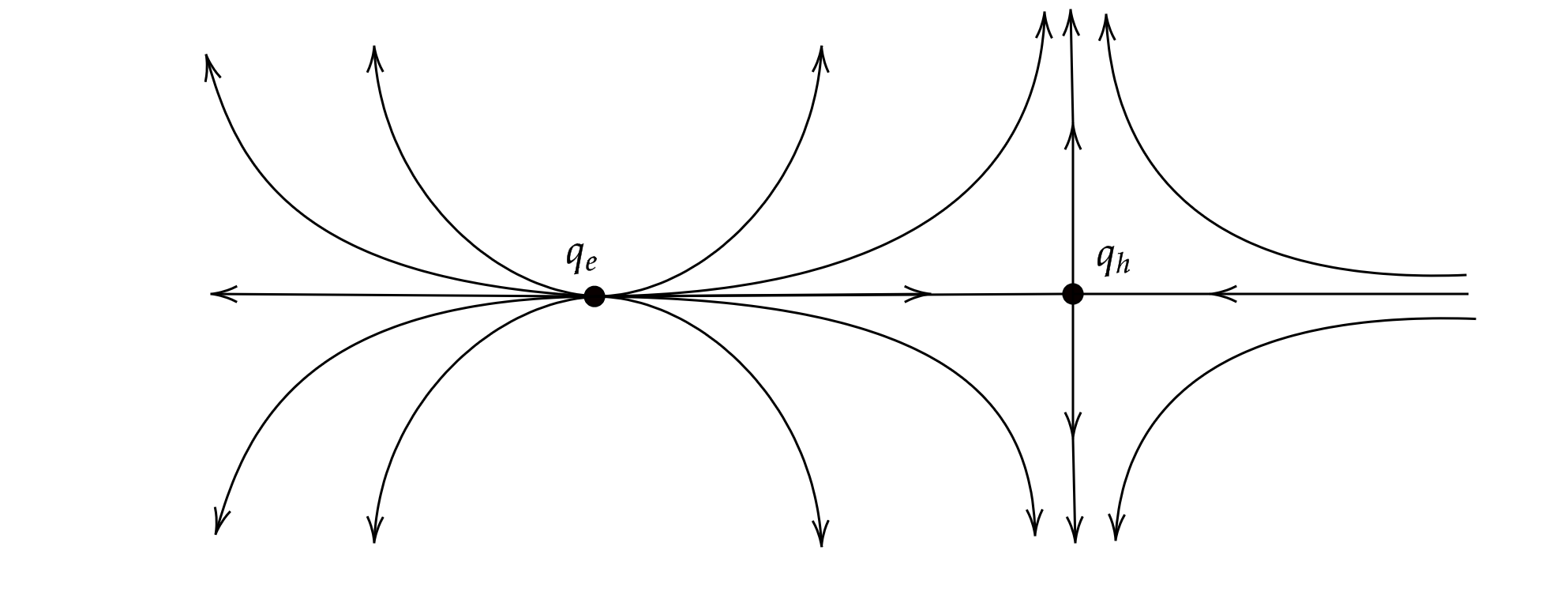}
\end{center}
\caption{An elliptic point $q_e$ and a saddle point $q_h$ in elimination position, as described by Lemma \ref{elimination_lem}.}
\label{b_el}
\end{figure}
The elimination lemma allows us to reduce the number of singularities in the characteristic foliation, without destroying the Morse-Smale property. Of course, this procedure can be reversed: we can use the elimination procedure to increase the number of singular points. In particular by means of an arbitrarily small $C^0$ perturbation of the surface $S$, it is possible to break any closed orbit introducing a pair of singular points along it, both positive if the orbit is repelling and negative if the orbit is attracting. We summarize this immediate consequence of the lemma in the following corollary.
\begin{figure}[H]
\begin{center}
    \includegraphics[width=11cm]{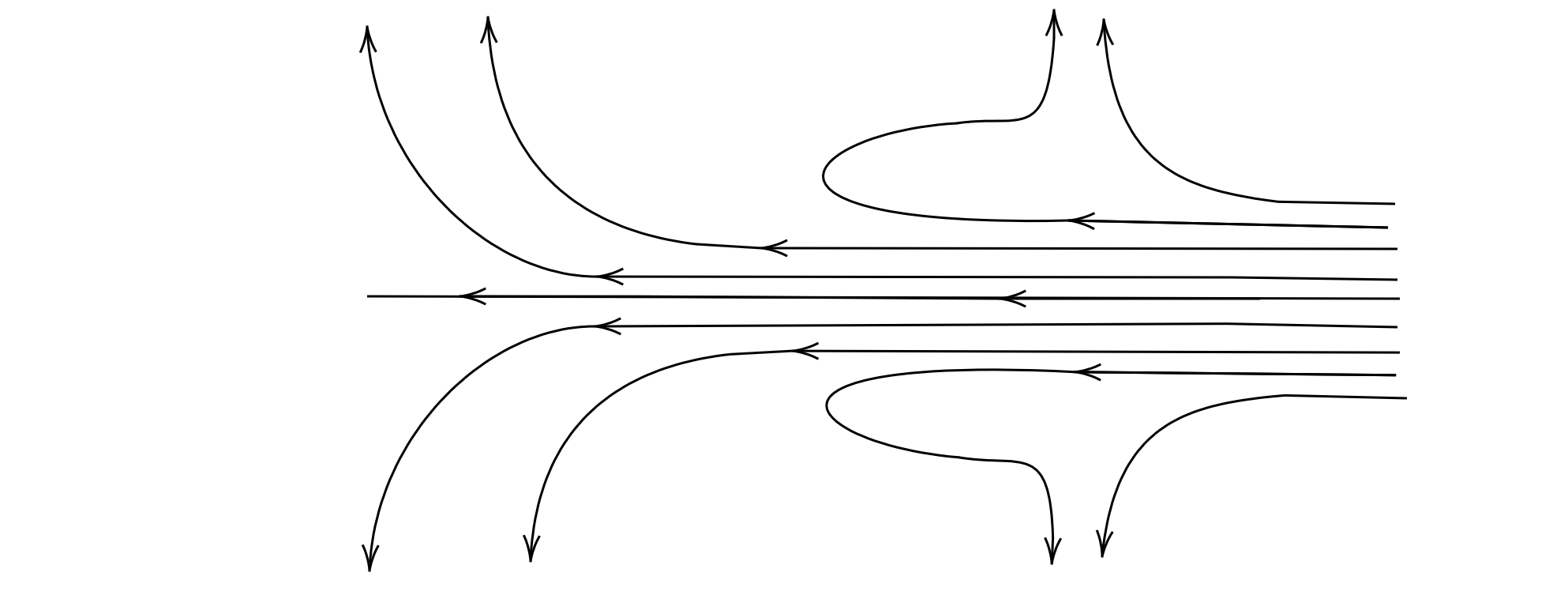}
\end{center}
\caption{Elimination of the critical point $q_e$, $q_h$ depicted in Figure \ref{b_el}}.
\label{a_el}
\end{figure}

\begin{corollary}\label{third-piece}
Let $S$ be a closed orientable surface embedded in a co-oriented contact three manifold $(M,\mathcal D)$, having characteristic foliation of Morse-Smale type. Then, there exists a surface $S'$ arbitrarily $C^0$-close to $S$, having a characteristic foliation of Morse-Smale type without closed orbits.
\end{corollary}
Putting together Proposition \ref{first-piece}, Lemma \ref{second-piece} and Corollary \ref{third-piece}, we deduce the following fact.
\begin{prop}\label{finite-lemma}
   Let $(M,\mathcal D, \bf{g})$ be a contact sub-Riemannian manifold, let $g$ be a positive integer and let $\bf{Emb}$$(S,M)$ be the space of embeddings of the genus $g$ orientable surface in $M$, endowed with the $C^0$ topology. The set of surfaces having finite induced distance is dense in $\bf{Emb}$$(S,M)$.
\end{prop}

\section{Proof of Theorem \ref{main}}\label{sec-U-infty}
  It is left to show that the subset of $\textbf{Emb}(S,M)$ of embedded surfaces with infinite induced distance is non empty, for any contact $3$-manifold $M$.
The result is a consequence of the following lemma, paired with Darboux theorem.
\begin{lemma}\label{infite-lemma}
    For every positive integer $g$ there exist a closed surface of genus $g$ embedded in the three dimensional Heisenberg group, having infinite induced distance.
\end{lemma}
\begin{proof}
The contact structure of the 3-dimensional Heisenberg group can be written as $(\mathbb R^3, \mathcal D=\text{span}\{\partial_x+\frac{y}{2}\partial_z, \partial_y-\frac{x}{2}\partial_z\})$. Let us denote the characteristic foliation of a surface $S$ embedded in $(\mathbb R^3, \mathcal D)$ by $\mathcal D S$. Consider first the case $g=1$. There exists an embedded torus in $(\mathbb R^3, \mathcal D)$ such that at least one leaf of $\mathcal DT$ is an embedded closed curve. Indeed given $0<r<R$, consider the torus $T_{R,r}$ given by the embedding 
\begin{equation}
\mathbb T^2\ni(\theta_1,\theta_2)\mapsto((R+r\cos\theta_1)\cos\theta_2, (R+r\cos\theta_1)\sin\theta_2, r\sin\theta_1).
\end{equation}
The characteristic foliation on $T_{R,r}$ is driven by the vector field
\begin{equation}\label{eq-X}
    X=\frac{(R+r\cos\theta_1)^2}{2}\partial_{\theta_1}-r\cos\theta_1\partial_{\theta_2}.
\end{equation}
Indeed we have 
\begin{equation*}
    \mathcal D=\ker\omega:=\ker\left\{dz+\frac{1}{2}(xdy-ydx)\right\}=\ker\left\{dz+\frac{\rho^2}{2}d\theta\right\},
\end{equation*}
where $\rho^2=x^2+y^2$ and $\theta=\arctan(y/x)$. Computing the kernel of the restriction 
\begin{equation*}
\omega_{|T_{r,R}}=r\cos\theta_1d\theta_1+\frac{(R+r\cos\theta_1)^2}{2}d\theta_2,
\end{equation*}
yields the filed \eqref{eq-X}.
Thanks to Proposition \ref{first-piece} we know that with a $C^\infty$ small perturbation of $T_{r,R}$ we can obtain a torus $T$ having a Morse-Smale characteristic foliation which can be assumed to be without singular points, since the field $X$ in \eqref{eq-X} is nowhere vanishing. A Morse-Smale characteristic foliation without singular points must contain closed orbits. This follows from the second property of Morse-Smale vector fields listed in Definition \ref{def-Morse}. The lemma is proved for $g=1$.
Let now $\Sigma_{g-1}$ be any embedded surface of genus $g-1$, not intersecting the torus $T$ mentioned above. Let $\ell\in\mathcal D T$ be a closed leaf and let $B\subset T$ be a ball not intersecting $\ell$, $\ell\cap B=\emptyset$. We define $\Sigma_{g}:=T\#\Sigma_{g-1}$, where the connected sum is made deleting $B$ from $T$, deleting a small ball from $\Sigma_{g-1}$, and gluing together (smoothly) the resulting boundary circles, see Figure \ref{connected_sum}. Since $B\cap\ell=\emptyset$, $\mathcal D\Sigma_{g}$ still has the closed curve $\ell$ among its leaves, thus $\Sigma_g$ has infinite induced distance.
\end{proof}
\begin{figure}[H]
\begin{center}
    \includegraphics[width=12cm]{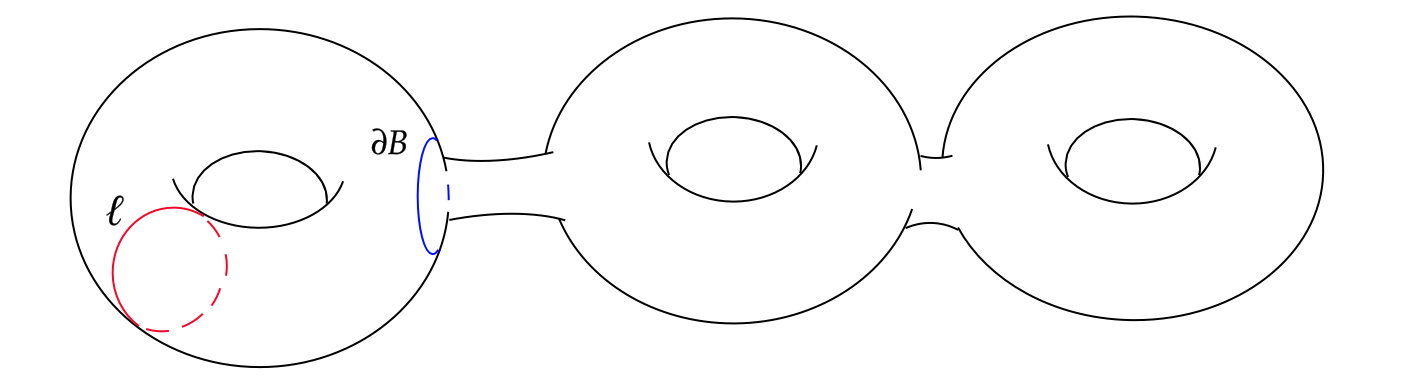}
\end{center}
\caption{Connected sum between the torus $T$ containing the closed leaf $\ell$ depicted in red and the boundary $\partial B$ depicted in blue, on the left, and the surface $\Sigma_{g-1}$, on the right.}
\label{connected_sum}
\end{figure}
We conclude the proof of Theorem \ref{main}. 
\begin{proof}
We define $\mathcal U_\infty, \mathcal U_f\subset\textbf{Emb}(S,M)$ to be set of embeddings of $S$ having a characteristic foliation of Morse-Smale type with and without closed orbits respectively. They are both open because of structural stability of Morse-Smale vector fields. Moreover, it follows from Proposition \ref{first-piece} that the union $\mathcal U_f\cup\mathcal U_\infty$ is dense in $\textbf{Emb}(S,M)$. Finally Lemma \ref{infite-lemma} ensures that $\mathcal U_\infty\neq\emptyset$, while Proposition \ref{finite-lemma} proves that $\mathcal U_f\neq\emptyset$.  
\end{proof}
{\bf Aknowledgements}\\[1mm]
This work has been partly supported by  the ANR-DFG projects “CoRoMo”
ANR-22-CE92-0077-01 and received financial support from the CNRS through the
MITI interdisciplinary programs 80 Prime.\newline
The authors would like thank Andrei A. Agrachev for very helpful discussions. 

\end{document}